\documentclass{amsart}

\usepackage{amsmath}
\usepackage{amsfonts}
\usepackage{color}
\usepackage{enumitem}

\newtheorem{thm}{Theorem}[section]

\newtheorem{lema}[thm]{Lemma}
\newtheorem{prop}[thm]{Proposition}

\theoremstyle{definition}
\newtheorem{defn}[thm]{Definition}
\theoremstyle{remark}
\newtheorem{rem}[thm]{Remark}

\numberwithin{equation}{section}
\newcommand{\A}{\mathcal A}
\newcommand{\R}{\mathbb R}
\newcommand{\N}{\mathbb N}

\newcommand{\s}{\mathbf{s}}
\newcommand{\p}{\mathbf{p}}
\newcommand{\K}{\mathcal K}
\newcommand{\D}{\mathcal{D}}

\def\diver{\mathop{\text{\normalfont div}}}

\DeclareMathOperator{\capa}{Cap}
\newcommand{\pv}{\mathrm{p.v.}}

\begin{document}
	\title[Shape optimization problems]{Shape optimization problems involving nonlocal and nonlinear operators}

	\author[Ignacio Ceresa Dussel]{Ignacio Ceresa Dussel}
	\address{Ignacio Ceresa Dussel\hfill\break\indent
		Departamento  de Matem\'atica, FCEyN, Universidad de Buenos Aires, \hfill\break\indent
		Instituto de C\'alculo, CONICET\hfill\break\indent
		Ciudad Universitaria, 0+$\infty$ building, C1428EGA, Av. Cantilo s/n\hfill\break\indent
		Buenos Aires, Argentina}




	\email{{\tt iceresad@dm.uba.ar}}


\begin{abstract}
	In this paper, we investigate a general shape optimization problem in which the state equation is given by a nonlocal and nonlinear operator.

	We prove the existence of a minimizer for a functional $F$ defined on the family of all quasi-open subsets of a bounded open set $\Omega\subset\mathbb{R}^n$, under a volume constraint. This holds provided that $F$ is decreasing with respect to set inclusion and lower semicontinuous with respect to a suitable topology associated with the fractional $p$-Laplacian with Dirichlet boundary conditions.

	Moreover, we study the asymptotic behavior of the solutions as $s\to1$, and we extend these results to the anisotropic case.

\end{abstract}
\maketitle
\section{Introduction}
In this paper, we address shape optimization problems. Problems of this type have been extensively studied in the literature; to mention only a few references, we refer the reader to the books by Bucur and Buttazzo \cite{BUTTAZZO-BUCUR}, Allaire \cite{ALLAIRE-02}, and Soko{\l}owski and Zol\'esio \cite{SOKOLOWSKI-ZOLESIO}.

In their most general form, these problems involve a \emph{cost functional} $F$ and a class of \emph{admissible domains} $\A$, and the objective is to solve the minimization problem
\begin{equation}\label{problem}
	\min_{A\in\A}F(A).
\end{equation}
Problems of this nature arise naturally in structural optimization, particularly when one wants to determine the optimal design of a system governed by partial differential equations.

In much of the existing literature, the cost functional $F$ is formulated in terms of a function $u_A$, which is the solution of a state equation posed on $A$, typically through an integral expression. Furthermore, these problems often include constraints on the Lipschitz constant of $\partial A$. However, such constraints are sometimes regarded as unnatural in practical applications. For this reason, in \cite{BUTTAZZO-DALMASO-93} the authors took as the family $\A$ of admissible domains all \emph{quasi-open} subsets of $\Omega$, with the sole constraint of a bound on their volume. Regarding the functional $F$, they assume that it is decreasing with respect to set inclusion and lower semicontinuous with respect to a suitable topology on $\A$, namely the $\gamma$-convergence: \emph{a sequence $\{A_k\}_{k\in \N}$ $\gamma$-converges to $A$ if for every $f\in H^{-1}(\Omega)$ the solutions of the Dirichlet problems
$$
-\Delta u_k=f \text{ in }A_k,\qquad u_k\in H^1_0(A_k),
$$ converge as $k\to \infty$ to the solution of the corresponding problem in $A$.}
In this way, the existence problem was solved in the linear case.

On the one hand, in \cite{BONDER-RITTORTO-18}, the authors study the \emph{nonlocal} linear problem, in which the state equation is the nonlocal Dirichlet problem
$$
(-\Delta)^su=f \text{ in }A,\qquad u=0 \text{ in }\R^n\setminus A,
$$
where $(-\Delta)^s$ is the \emph{fractional Laplacian} and $f$ belongs to a suitable space.

On the other hand, in \cite{FUSCO-19}, the authors prove similar results when the state equation involves the local nonlinear $p$-Laplacian
$$
-\Delta_pu:=-\diver(|\nabla u|^{p-2}\nabla u).
$$
Moreover, they prove a minimax characterization of the second eigenvalue of the $p$-Laplacian on $p$-quasi-open sets, and in \cite{MAZZOLENI-21} the authors prove the existence of an optimal set for the minimization of the $k$-th variational eigenvalue of the $p$-Laplacian.

In this paper, we first focus on the natural next step, namely minimization problems of the form \eqref{problem} in which the relation between the cost functional and the domain is given by the nonlocal and nonlinear Dirichlet problem
$$
(-\Delta_p)^su=f \text{ in }A,\qquad u=0 \text{ in }\R^n\setminus A,
$$
where the operator is the \emph{fractional $p$-Laplacian}
$$
(-\Delta_p)^su(x):=2c_{n,p,s}\,\pv\int_{\R^n}\frac{|u(x)-u(y)|^{p-2}(u(x)-u(y))}{|x-y|^{n+sp}}\,dy,
$$
where $c_{n,p,s}:=p(1-s)/K(n,p)$ and $K(n,p):=\int_{\mathbb{S}^{n-1}}|\sigma\cdot e_1|^p\,d\mathcal{H}^{n-1}(\sigma)$. The normalization constant $c_{n,p,s}$ is chosen so that, by the results of Bourgain, Brezis and Mironescu \cite{BBM-01}, the associated energies converge to $\frac1p\int|\nabla u|^p\,dx$ as $s\uparrow1$ (see Section \ref{sec_preliminares}). Keep in mind that the fractional $p$-Laplacian is the nonlocal counterpart of the $p$-Laplacian. Based on \cite{BBM-01,BONDER-SALORT-20,BONDER-RITTORTO-18}, we examine the transition from nonlocal to local operators, specifically studying the asymptotic behavior as $s$ approaches 1.

Finally, based on \cite{CERESABONDER23,CHAKER-23}, we introduce the anisotropic counterparts of the $p$-Laplacian and the fractional $p$-Laplacian, namely the \emph{pseudo $p$-Laplacian}
$$
-\widetilde{\Delta}_\p u:=-\sum_{i=1}^{n}(|u_{x_i}|^{p_i-2}u_{x_i})_{x_i},
$$
and the \emph{fractional pseudo $p$-Laplacian}
\begin{equation}\label{fractiona pseudo p}
(-\widetilde{\Delta}_\p)^\s u(x):=2\sum_{i=1}^{n}(1-s_i)s_i\,\pv\!\int_{\R}\frac{|u(x)-u(x+he_i)|^{p_i-2}(u(x)-u(x+he_i))}{|h|^{1+s_ip_i}}\,dh.
\end{equation}
In this expression, $\s=(s_1,\dots,s_n)$ are the fractional differentiability parameters and $\p=(p_1,\dots,p_n)$ are the integrability parameters.

The paper is structured as follows. In Section \ref{sec_preliminares}, we present the preliminaries and some useful properties. In Section \ref{sec_teo1}, we state and prove the main result, which concerns the existence of a minimizer. Section \ref{sec_asint} discusses the asymptotic behavior as $s$ approaches 1. Finally, in Section \ref{sec_anis} we state the corresponding results for anisotropic operators.
%

\section{Preliminaries}\label{sec_preliminares}
In this section, we introduce the definitions and properties of the functional framework that will be used throughout the paper.

Let $0 < s < 1$ and $1 < p < \infty$ be the regularity and integrability parameters. The fractional Sobolev space is defined as
\[
W^{s, p}(\R^n) = \left\{ u \in L^p(\R^n) : \frac{u(x) - u(y)}{|x - y|^{\frac{n}{p} + s}} \in L^p(\R^n\times\R^n) \right\}.
\]
This space is endowed with the norm
$$
\|u\|_{W^{s, p}(\R^n)} = \left( \|u\|_{L^p(\R^n)}^p + [u]_{s, p}^p \right)^{\frac{1}{p}},
$$
where
$$
[u]_{s,p} = \left( \iint_{\R^n \times \R^n} \frac{|u(x) - u(y)|^p}{|x - y|^{n+sp}} \, dxdy \right)^{\frac{1}{p}
}.
$$
With this norm, \(W^{s, p}(\R^n)\) is a reflexive Banach space.

If $\Omega\subset\R^n$ is a bounded open set, a common practice is to consider the closure of $C^\infty_c(\Omega)$ with respect to $\|\cdot\|_{W^{s,p}(\R^n)}$. Another definition is given by
	$$
	W_0^{s, p}(\Omega)=\{u\in W^{s, p}(\R^n)\colon u=0 \text{ a.e. in }  \R^n\setminus\Omega \}.
	$$
	It is well known that these two definitions agree when, for example, $\partial \Omega$ is Lipschitz. Since the regularity of $\partial\Omega$ plays no role in our problem, we assume throughout the paper that $\Omega$ is a bounded open set with Lipschitz boundary.

When $s=1$, we consider the usual Sobolev space $W^{1,p}_0(\Omega)$, with the norm $[u]_{1,p}:=\|\nabla u\|_{L^p(\Omega)}$.
\begin{defn}
	Given $0 < s \leq 1 < p < \infty$, the (topological) dual space of $W^{s,p}_0(\Omega)$ is denoted by $W^{-s,p'}(\Omega)$, and $\langle\cdot,\cdot\rangle$ denotes the duality pairing between $W^{-s,p'}(\Omega)$ and $W^{s,p}_0(\Omega)$.

\end{defn}

\subsection{Capacity}
In \cite{WARMA-15} the author defines the $(s,p)$-capacity of a set $A\subset\Omega$. We recall some useful definitions and results about this fractional capacity, and refer the interested reader to \cite{WARMA-15} and the references therein; for the case $s=1$, see \cite{HKM}.

\begin{defn}
	Let $\Omega \subset \mathbb{R}^n$ be an open set. Given $A \subset \Omega$ and $0 < s \leq 1<p<\infty$, we define the $(s,p)$-capacity of $A$ as
	$$
	\capa(s,p)(A) = \inf \left\{ [u]_{s,p} \colon u \in W^{s,p}_0(\Omega),\ u \geq 0,\ A \subset \{u \geq 1\}^\circ \right\}.
	$$
\end{defn}
If a property $P(x)$ holds for every $x \in E$ except for those in a set $Z \subset E$ with $\capa(s,p)(Z) = 0$, we say that $P(x)$ holds $(s,p)$-quasi everywhere on $E$ (abbreviated as q.e. on $E$).
\begin{defn}
A function $u \colon \Omega \rightarrow \R$ is said to be quasi-continuous (q.c.) if for every $\varepsilon > 0$ there exists a relatively open set $G \subset \Omega$ such that $\capa(s,p) (G) < \varepsilon$ and $u|_{\Omega \setminus G}$ is continuous.
\end{defn}
\begin{thm}[{\cite[Theorem 3.7]{WARMA-15}}]
	For every $u \in W^{s,p}_0(\Omega)$ there exists a quasi-continuous function $\tilde{u} \colon \Omega \rightarrow \mathbb{R}$ such that $\tilde{u} = u$ a.e. in $\Omega$, which is unique up to sets of zero capacity.
\end{thm}
In what follows, we always identify a function $u\in W^{s,p}_0(\Omega)$ with its quasi-continuous representative $\tilde u$.

\begin{defn}\label{def_qo}
	A set $A\subset\Omega$ is said to be \emph{quasi-open} if for every $\varepsilon>0$ there exists an open set $A^\varepsilon\subset\Omega$ such that $\capa(s,p)(A\triangle A^\varepsilon)\leq\varepsilon$, where $\triangle$ denotes the symmetric difference of sets.
\end{defn}

\begin{rem}\label{rem_capacity}
We will use the following standard facts (see \cite{WARMA-15} and, for $s=1$, \cite{HKM}).
\begin{enumerate}[label=(\alph*)]
	\item By definition, $\capa(s,p)$ is outer regular: for every $E\subset\Omega$ and $\delta>0$ there exists an open set $G\supset E$ with $\capa(s,p)(G)<\capa(s,p)(E)+\delta$. Consequently, $A$ is quasi-open if and only if for every $\varepsilon>0$ there exists an open set $\omega_\varepsilon$ with $\capa(s,p)(\omega_\varepsilon)<\varepsilon$ such that $A\cup\omega_\varepsilon$ is open. In particular, finite unions of quasi-open sets are quasi-open.
	\item If $u_k\to u$ in $W^{s,p}_0(\Omega)$, then a subsequence of $\{u_k\}_{k\in\N}$ converges to $u$ q.e. in $\Omega$.
	\item If $u,v\in W^{s,p}_0(\Omega)$ and $u\le v$ a.e. in $\Omega$, then $u\le v$ q.e. in $\Omega$.
	\item If $u\in W^{s,p}_0(\Omega)$ and $t\in\R$, then $\{u>t\}$ is quasi-open.
\end{enumerate}
\end{rem}

We denote by $\A_{s,p}(\Omega)$ the class of all $(s,p)$-quasi-open subsets of $\Omega$ and, for every $A\in \A_{s,p}(\Omega)$, we define the space
$$
W_0^{s, p}(A)=\{u\in W^{s, p}_0(\Omega)\colon u=0 \text{ q.e. in }  \Omega\setminus A \}.
$$
By Remark \ref{rem_capacity}(b), $W_0^{s, p}(A)$ is a closed subspace of $W^{s,p}_0(\Omega)$; therefore, it is a reflexive Banach space. If $A$ is open, this definition of $W^{s,p}_0(A)$ is equivalent to the usual one.

\subsection{The fractional $p$-Laplacian}
For $0<s<1$ we consider the normalized energy
\begin{equation}\label{J}
J_{s,p}(u)=\frac{c_{n,p,s}}{p}\iint_{\R^n\times\R^n}\frac{|u(x)-u(y)|^p}{|x-y|^{n+sp}}\,dx\,dy=\frac{1-s}{K(n,p)}[u]_{s,p}^p,\qquad u\in W^{s,p}_0(\Omega),
\end{equation}
where $c_{n,p,s}=p(1-s)/K(n,p)$ and $K(n,p)$ are as in the introduction, and for $s=1$ we set $J_{1,p}(u)=\frac1p\int_\Omega|\nabla u|^p\,dx$. By \cite{BBM-01}, $J_{s,p}(u)\to J_{1,p}(u)$ as $s\uparrow1$ for every $u\in W^{1,p}_0(\Omega)$. The functional $J_{s,p}$ is convex and continuous on $W^{s,p}_0(\Omega)$, with continuous Gateaux derivative $J_{s,p}'\colon W^{s,p}_0(\Omega)\to W^{-s,p'}(\Omega)$ given, for $s<1$, by
\begin{equation}\label{J'}
\langle J_{s,p}'(u),v\rangle= c_{n,p,s}\iint_{\R^n \times \R^n} \frac{|u(x) - u(y)|^{p-2}(u(x)-u(y))(v(x)-v(y))}{|x - y|^{n+sp}} \, dx\, dy,
\end{equation}
and, for $s=1$, by $\langle J_{1,p}'(u),v\rangle=\int_{\Omega}|\nabla u|^{p-2}\nabla u\cdot\nabla v \,dx$.
Note that $J'_{s,p}(tu)=|t|^{p-2}t\,J'_{s,p}(u)$ for $t\in\R$ and that $\langle J'_{s,p}(u),u\rangle=pJ_{s,p}(u)$. We write $\langle(-\Delta_p)^su,v\rangle:=\langle J'_{s,p}(u),v\rangle$; for smooth $u$ this coincides with the pointwise definition given in the introduction (and with $-\Delta_pu$ when $s=1$). Finally, by the Poincar\'e inequality, $J_{s,p}(u)^{1/p}$ is a norm on $W^{s,p}_0(\Omega)$ equivalent to $\|\cdot\|_{W^{s,p}(\R^n)}$.

\begin{defn}
	Let $0 < s \leq 1 < p < \infty$, $\mathcal{B}\in\A_{s,p}(\Omega)$ and $f \in L^{p'}(\Omega)$, where $p'$ is the conjugate exponent of $p$. We say that \( u \in W^{s,p}_0(\mathcal{B}) \) is a weak solution of
	\begin{equation}\label{weak-solution}
	\begin{cases}
	(-\Delta_p)^s u = f &\text{in } \mathcal{B}, \\
	u = 0 &\text{q.e. in } \Omega \setminus \mathcal{B},
	\end{cases}
	\end{equation}
	if
	\[
	\langle (-\Delta_p)^s u , v\rangle = \int_\mathcal{B} f v \, dx\quad\text{for every } v \in W^{s,p}_0(\mathcal{B}).
	\]
	Similarly, given $g\in L^{p'}(\Omega)$, we say that $(-\Delta_p)^su\le g$ in $\mathcal B$ if $\langle(-\Delta_p)^su,v\rangle\le\int_{\mathcal B}gv\,dx$ for every $v\in W^{s,p}_0(\mathcal B)$ with $v\ge0$; the inequality $(-\Delta_p)^su\ge g$ is defined analogously.
\end{defn}
Since $J_{s,p}$ is strictly convex, coercive and weakly lower semicontinuous on the closed subspace $W^{s,p}_0(\mathcal B)$, problem \eqref{weak-solution} has a unique weak solution, namely the unique minimizer of $J_{s,p}(v)-\int_{\mathcal B}fv\,dx$ in $W^{s,p}_0(\mathcal B)$. In particular, given $A \in \A_{s,p}(\Omega)$, we denote by $u_A^s \in W_0^{s,p}(A)$ the unique weak solution of
\begin{equation}\label{def_u^a}
(-\Delta_p)^s v = 1 \text{ in } A, \qquad v = 0 \text{ q.e. in } \Omega\setminus A,
\end{equation}
that is, the unique minimizer of
$$
I_s(v):=J_{s,p}(v)-\int_\Omega v\,dx
$$
in $W^{s,p}_0(A)$. When $s=1$ we simply write $u_A:=u^1_A$.

We will use the following comparison principle and stability estimate. The comparison principle is well known; see \cite{LINDQVIST-14} for $s<1$. We include the short proof, since we need it on quasi-open sets and in the form below.
\begin{lema}[Comparison principle]\label{lemma_comparison}
	Let $0<s\le1$ and $u,v\in W^{s,p}_0(\Omega)$. Then $\langle J'_{s,p}(u)-J'_{s,p}(v),(u-v)^+\rangle\ge0$, and if $\langle J'_{s,p}(u)-J'_{s,p}(v),(u-v)^+\rangle\le0$, then $u\le v$ a.e. in $\R^n$. In particular, if $A\in\A_{s,p}(\Omega)$, $(u-v)^+\in W^{s,p}_0(A)$ and $(-\Delta_p)^su\le(-\Delta_p)^sv$ in $A$, in the sense that $\langle J'_{s,p}(u),\varphi\rangle\le\langle J'_{s,p}(v),\varphi\rangle$ for every $\varphi\in W^{s,p}_0(A)$ with $\varphi\ge0$, then $u\le v$ a.e. in $\R^n$.
\end{lema}
\begin{proof}
	Let $\psi=u-v$ and $\phi_p(t)=|t|^{p-2}t$. Assume first that $s<1$. For $x,y\in\R^n$, write $a=u(x)-u(y)$ and $b=v(x)-v(y)$, so that $a-b=\psi(x)-\psi(y)$. Since $\phi_p$ is strictly increasing and $t\mapsto t^+$ is nondecreasing, both $\phi_p(a)-\phi_p(b)$ and $\psi^+(x)-\psi^+(y)$ have the sign of $\psi(x)-\psi(y)$ whenever they do not vanish. Hence the integrand of
	$$
	\langle J'_{s,p}(u)-J'_{s,p}(v),\psi^+\rangle=c_{n,p,s}\iint_{\R^n\times\R^n}\frac{(\phi_p(a)-\phi_p(b))(\psi^+(x)-\psi^+(y))}{|x-y|^{n+sp}}\,dx\,dy
	$$
	is nonnegative. If this integral is nonpositive, the integrand vanishes a.e.; since $\phi_p(a)=\phi_p(b)$ implies $\psi(x)=\psi(y)$, we get $\psi^+(x)=\psi^+(y)$ for a.e. $(x,y)\in\R^n\times\R^n$. Thus $\psi^+$ is constant a.e. in $\R^n$ and, since it vanishes outside $\Omega$, $\psi^+=0$. If $s=1$, then
	$$
	\langle J'_{1,p}(u)-J'_{1,p}(v),\psi^+\rangle=\int_{\{\psi>0\}}(|\nabla u|^{p-2}\nabla u-|\nabla v|^{p-2}\nabla v)\cdot(\nabla u-\nabla v)\,dx\ge0,
	$$
	and equality forces $\nabla\psi^+=0$ a.e., hence $\psi^+=0$. The last statement follows by taking $\varphi=(u-v)^+$.
\end{proof}

\begin{lema}\label{lemma_stability}
	Let $0<s\le1$, $A\in\A_{s,p}(\Omega)$, and let $C_P>0$ be such that $\|z\|_{L^p(\Omega)}^p\le C_PJ_{s,p}(z)$ for every $z\in W^{s,p}_0(\Omega)$. Let $g_1,g_2\in L^{p'}(\Omega)$ and let $z_1,z_2\in W^{s,p}_0(A)$ be the weak solutions of $(-\Delta_p)^sz_i=g_i$ in $A$, $i=1,2$. Then $J_{s,p}(z_i)\le C\|g_i\|_{L^{p'}(\Omega)}^{p'}$ and
	$$
	J_{s,p}(z_1-z_2)\le\rho\big(\|g_1-g_2\|_{L^{p'}(\Omega)}\big),
	$$
	where $\rho(t)=Ct^{p'}$ if $p\ge2$ and $\rho(t)=Ct^p$ if $1<p<2$. Here $C$ depends only on $n$, $p$, $C_P$ and, when $1<p<2$, on an upper bound for $J_{s,p}(z_1)+J_{s,p}(z_2)$; in particular, $C$ does not depend on $s$ as long as $C_P$ does not.
\end{lema}
\begin{proof}
	Testing the equation of $z_i$ with $z_i$ and using H\"older's inequality we get $pJ_{s,p}(z_i)=\int_\Omega g_iz_i\,dx\le\|g_i\|_{p'}(C_PJ_{s,p}(z_i))^{1/p}$, which gives the first estimate. Let $\delta=z_1-z_2\in W^{s,p}_0(A)$. Testing both equations with $\delta$ we get
	\begin{equation}\label{eq_stab}
	\langle J'_{s,p}(z_1)-J'_{s,p}(z_2),\delta\rangle=\int_\Omega(g_1-g_2)\delta\,dx\le\|g_1-g_2\|_{p'}\big(C_PJ_{s,p}(\delta)\big)^{1/p}.
	\end{equation}
	We use the well-known inequality
	\begin{equation}\label{desig}
		(|a|^{p-2}a-|b|^{p-2}b)\cdot(a-b)\geq \begin{cases}
			c_1(|a|+|b|)^{p-2}|a-b|^2&\text{if }1<p<2,\\
			c_2|a-b|^p&\text{if }p\geq 2,
		\end{cases}
	\end{equation}
	valid for $a,b\in\R^m$, where $c_1,c_2>0$ depend only on $p$. Assume that $s<1$, and write $d\mu=c_{n,p,s}|x-y|^{-n-sp}\,dx\,dy$, $a=z_1(x)-z_1(y)$ and $b=z_2(x)-z_2(y)$. If $p\ge2$, \eqref{desig} gives $c_2pJ_{s,p}(\delta)\le\langle J'_{s,p}(z_1)-J'_{s,p}(z_2),\delta\rangle$, which combined with \eqref{eq_stab} yields $J_{s,p}(\delta)^{1/p'}\le C\|g_1-g_2\|_{p'}$. If $1<p<2$, by H\"older's inequality with exponents $2/p$ and $2/(2-p)$,
	$$
	pJ_{s,p}(\delta)=\iint|a-b|^p\,d\mu\le\Big(\iint(|a|+|b|)^{p-2}|a-b|^2\,d\mu\Big)^{\frac p2}\Big(\iint(|a|+|b|)^p\,d\mu\Big)^{\frac{2-p}{2}},
	$$
	and $\iint(|a|+|b|)^p\,d\mu\le2^{p-1}p\,(J_{s,p}(z_1)+J_{s,p}(z_2))$. Hence, by \eqref{desig} and \eqref{eq_stab}, $J_{s,p}(\delta)\le C\big(\|g_1-g_2\|_{p'}J_{s,p}(\delta)^{1/p}\big)^{p/2}$, that is, $J_{s,p}(\delta)\le C\|g_1-g_2\|_{p'}^p$. The case $s=1$ is analogous, replacing $a$, $b$ by $\nabla z_1$, $\nabla z_2$ and $d\mu$ by $dx$.
\end{proof}

We will employ the following concept of set convergence, which is the fractional version of the $\gamma$-convergence of sets \cite{BUTTAZZO-DALMASO-93,DALMASO-93}.
\begin{defn}\label{def gamma}
	Let $\{A_k\}_{k\in\mathbb{N}} \subset \A_{s,p}(\Omega)$ and $A \in \A_{s,p}(\Omega)$. We say that $A_k \xrightarrow{\gamma} A$ if $u^s_{A_k} \to u_A^s$ in $L^p(\Omega)$.
\end{defn}
Now we show some properties of $u_A^s$.
\begin{lema}\label{lemma2}
	Let $0<s\le1$ and $A\in \A_{s,p}(\Omega)$. Then $u^s_A\ge0$ and
	$$
	u_A^s=\max\{w\in W^{s,p}_0(\Omega)\colon w\leq 0 \text{ q.e. in } \Omega\setminus A\text{ and } (-\Delta_p)^sw\leq 1 \text{ in } \Omega\}.
	$$
	In particular, if $A,B\in\A_{s,p}(\Omega)$ and $A\subset B$, then $u^s_A\le u^s_B$.
\end{lema}
\begin{proof}
	Let us define the convex set
	$$	K_A=\{w\in W^{s,p}_0(\Omega)\colon w\leq 0 \text{ q.e. in } \Omega\setminus A	\},
	$$
	which is closed by Remark \ref{rem_capacity}(b), and let $w_A$ be the unique minimizer of $I_s$ on $K_A$. Then $w_A$ satisfies the variational inequality
	\begin{equation}\label{variational_equation}
	\langle J_{s,p}'(w_A) , v-w_A\rangle
	\geq \int_{\Omega} (v-w_A)\,dx\quad\forall v\in K_A.
	\end{equation}
	Taking $v=w_A^+\in K_A$ in \eqref{variational_equation}, so that $v-w_A=w_A^-$, we obtain
	\begin{align*}
		0\leq\int_{\Omega}w_A^-\,dx \leq \langle J_{s,p}'(w_A) , w_A^-\rangle\leq
		-\langle J'_{s,p}(w_A^-),w_A^-\rangle=-pJ_{s,p}(w_A^-),
	\end{align*}
	where the last inequality follows from the elementary inequality $|a-b|^{p-2}(a-b)\,b\le-|b|^p$, valid whenever $ab\le0$, applied with $a=w_A^+(x)-w_A^+(y)$ and $b=w_A^-(x)-w_A^-(y)$ (when $s=1$ it is an equality). Hence $J_{s,p}(w_A^-)=0$, that is, $w_A^-=0$, so $w_A\ge0$. Since $w_A\in K_A$, it follows that $w_A=0$ q.e. in $\Omega\setminus A$, i.e., $w_A\in W^{s,p}_0(A)$. As $W^{s,p}_0(A)\subset K_A$ and $u_A^s$ is the unique minimizer of $I_s$ in $W^{s,p}_0(A)$, we conclude that $w_A=u_A^s$; in particular, $u^s_A\ge0$.

	Next, given $v \in W^{s,p}_0(\Omega)$ with $v \geq 0$, we have $-v \in K_A$. Hence, applying \eqref{variational_equation},
	\begin{align}\label{eq6}
		\langle J_{s,p}'( u_A^s) ,-v\rangle+\langle J_{s,p}'( u_A^s) ,-  u_A^s\rangle=\langle J_{s,p}'( u_A^s) ,-v-  u_A^s\rangle \geq\int_{\Omega} (-v- u_A^s)\,dx.
	\end{align}
	Since $\langle J_{s,p}'( u_A^s) ,u_A^s\rangle = \int_{\Omega} u_A^s\,dx$, inequality \eqref{eq6} reads
	$$
	\langle J_{s,p}'( u_A^s) ,v\rangle \leq\int_{\Omega} v\,dx.
	$$
	Since $v\ge0$ is arbitrary, $(-\Delta_p)^s u_A^s\leq 1$ in $\Omega$.

	Finally, let $w\in W^{s,p}_0(\Omega)$ be such that $w\leq 0$ q.e. in $\Omega \setminus A$ and $(-\Delta_p)^sw\leq 1$ in $\Omega$. Then $(w-u^s_A)^+\in W^{s,p}_0(A)$ and, for every $\varphi\in W^{s,p}_0(A)$ with $\varphi\ge0$,
	$$
	\langle J'_{s,p}(w),\varphi\rangle\le\int_\Omega\varphi\,dx=\langle J'_{s,p}(u^s_A),\varphi\rangle.
	$$
	By Lemma \ref{lemma_comparison}, $w\leq  u_A^s$. If $A\subset B$, then $u^s_A$ is admissible in the characterization of $u^s_B$, so $u^s_A\le u^s_B$.
\end{proof}
To conclude this section, we prove two lemmas that are crucial for our purposes. In the proofs we use the elementary inequality
\begin{equation}\label{submodular}
|\max\{a_1,b_1\}-\max\{a_2,b_2\}|^p+|\min\{a_1,b_1\}-\min\{a_2,b_2\}|^p\le|a_1-a_2|^p+|b_1-b_2|^p,
\end{equation}
valid for $a_1,a_2,b_1,b_2\in\R$. (If $(a_1-b_1)(a_2-b_2)\ge0$, it is an equality; otherwise, say $a_1>b_1$ and $a_2<b_2$, the numbers $a_1-b_2$ and $b_1-a_2$ have the same sum as $a_1-a_2$ and $b_1-b_2$ and lie between them, so \eqref{submodular} follows from the convexity of $t\mapsto|t|^p$.) Consequently, $J_{s,p}(\min\{u,v\})+J_{s,p}(\max\{u,v\})\le J_{s,p}(u)+J_{s,p}(v)$, and the same holds for any functions $u,v$ for which the right-hand side, defined by \eqref{J}, is finite.
\begin{lema}\label{lemaA}
	Let $0<s\le1$ and let $\{A_k\}_{k\in\N}\subset\A_{s,p}(\Omega)$ be such that $u_{A_k}^s$ converges weakly in $W^{s,p}_0(\Omega)$ to a function $u$. Let $\{v_k\}_{k\in \N}\subset W^{s,p}_0(\Omega)$ be such that $v_k\in W^{s,p}_0(A_k)$ for every $k\in\N$, and assume that $v_k\rightharpoonup v$ weakly in $W^{s,p}_0(\Omega)$. Then $v=0$ q.e. in $\{u=0\}$.
\end{lema}
\begin{proof}
	Throughout the proof we write $J=J_{s,p}$, $u_k=u^s_{A_k}$ and $\|\cdot\|_q=\|\cdot\|_{L^q(\Omega)}$. Since the embedding $W^{s,p}_0(\Omega)\hookrightarrow L^p(\Omega)$ is compact, $u_k\to u$ and $v_k\to v$ strongly in $L^p(\Omega)$; note also that $u\ge0$ by Lemma \ref{lemma2}. Let
	$$
	\mathcal C=\{z\in W^{s,p}_0(\Omega)\colon z=0\text{ q.e. in }\{u=0\}\}.
	$$
	By Remark \ref{rem_capacity}(b), $\mathcal C$ is a closed convex subset of $W^{s,p}_0(\Omega)$, hence it is weakly closed. We have to show that $v\in\mathcal C$, and we may freely pass to subsequences. Let $L=\sup_kJ(v_k)<\infty$.

	\emph{Step 1: penalization.} Fix $\lambda>0$ and let $\hat w_k$ be the unique minimizer in $W^{s,p}_0(A_k)$ of
	$$
	\Phi_k(z)=J(z)+\frac{1}{p\lambda}\|z-v\|_p^p.
	$$
	Comparing with $v_k$, we get $\Phi_k(\hat w_k)\le J(v_k)+\frac{1}{p\lambda}\|v_k-v\|_p^p$, so $\{\hat w_k\}_{k\in\N}$ is bounded in $W^{s,p}_0(\Omega)$ and, up to a subsequence, $\hat w_k\rightharpoonup w^\lambda$ weakly in $W^{s,p}_0(\Omega)$ and strongly in $L^p(\Omega)$. By weak lower semicontinuity,
	\begin{equation}\label{eq_wlambda}
	J(w^\lambda)+\frac{1}{p\lambda}\|w^\lambda-v\|_p^p\le\liminf_{k\to\infty}\Phi_k(\hat w_k)\le L.
	\end{equation}
	Moreover, $\hat w_k$ is the weak solution of $(-\Delta_p)^s\hat w_k=f_k$ in $A_k$, where $f_k=-\lambda^{-1}|\hat w_k-v|^{p-2}(\hat w_k-v)$, and $f_k\to f:=-\lambda^{-1}|w^\lambda-v|^{p-2}(w^\lambda-v)$ strongly in $L^{p'}(\Omega)$.

	\emph{Step 2: $w^\lambda\in\mathcal C$.} Given $\delta>0$, choose $f_\delta\in L^\infty(\Omega)$ with $\|f_\delta-f\|_{p'}\le\delta$, and let $\omega_{k,\delta}\in W^{s,p}_0(A_k)$ be the weak solution of $(-\Delta_p)^s\omega_{k,\delta}=f_\delta$ in $A_k$. By Lemma \ref{lemma_stability}, $\{\omega_{k,\delta}\}_{k\in\N}$ is bounded in $W^{s,p}_0(\Omega)$ and
	$$
	J(\omega_{k,\delta}-\hat w_k)\le\rho(\|f_\delta-f_k\|_{p'}),
	$$
	where $\rho$ does not depend on $k$. Up to a subsequence, $\omega_{k,\delta}\rightharpoonup\omega_\delta$ weakly in $W^{s,p}_0(\Omega)$ and, by weak lower semicontinuity,
	\begin{equation}\label{eq_omegadelta}
	J(\omega_\delta-w^\lambda)\le\liminf_{k\to\infty}\rho(\|f_\delta-f_k\|_{p'})=\rho(\|f_\delta-f\|_{p'})\le\rho(\delta).
	\end{equation}
	On the other hand, let $c_\delta=\|f_\delta\|_{L^\infty(\Omega)}^{1/(p-1)}$. Since $(-\Delta_p)^s(\pm\omega_{k,\delta})=\pm f_\delta\le c_\delta^{p-1}=(-\Delta_p)^s(c_\delta u_k)$ in $A_k$, Lemma \ref{lemma_comparison} gives $|\omega_{k,\delta}|\le c_\delta u_k$ a.e. Passing to the limit (the convergences being strong in $L^p(\Omega)$), $|\omega_\delta|\le c_\delta u$ a.e., hence q.e. by Remark \ref{rem_capacity}(c). Thus $\omega_\delta\in\mathcal C$. By \eqref{eq_omegadelta}, $\omega_\delta\to w^\lambda$ in $W^{s,p}_0(\Omega)$ as $\delta\to0$, and since $\mathcal C$ is closed, $w^\lambda\in\mathcal C$.

	\emph{Step 3: conclusion.} By \eqref{eq_wlambda}, $\|w^\lambda-v\|_p^p\le p\lambda L$ and $J(w^\lambda)\le L$. Hence $w^\lambda\rightharpoonup v$ weakly in $W^{s,p}_0(\Omega)$ as $\lambda\to0$ and, since $\mathcal C$ is weakly closed, $v\in\mathcal C$.
\end{proof}

\begin{lema}\label{lemmaB}
	Let $0<s<1$, $A\in\A_{s,p}(\Omega)$ and $\{A_k\}_{k\in \N}\subset\A_{s,p}(\Omega)$ be such that $u^s_{A_k}\rightharpoonup u$ weakly in $W^{s,p}_0(\Omega)$, with $u\leq  u_A^s$ in $\Omega$. Given $\varepsilon>0$, let $A^{\varepsilon}=\{ u_A^s>\varepsilon\}$ and assume that $u^s_{A_k\cup A^\varepsilon}$ converges weakly in $W^{s,p}_0(\Omega)$ to some function $u^\varepsilon$. Then $u^\varepsilon\leq  u_A^s$ q.e. in $\Omega$.
\end{lema}
\begin{proof}
	Note that $A^\varepsilon$ and $A_k\cup A^\varepsilon$ are quasi-open by Remark \ref{rem_capacity}. Let $v_\varepsilon=1-\frac{1}{\varepsilon}\min\{ u_A^s,\varepsilon\}$. Then $1-v_\varepsilon\in W^{s,p}_0(\Omega)$ (note that $v_\varepsilon$ itself does not vanish outside $\Omega$), $0\leq v_\varepsilon\leq 1$, $v_\varepsilon=0$ in $A^\varepsilon$ and $v_\varepsilon=1$ q.e. in $\Omega\setminus A$.

	Now set $w_k=\min\{v_\varepsilon,u^s_{A_k\cup A^\varepsilon}\}$. Since $u^s_{A_k\cup A^\varepsilon}\ge0$, we have $w_k=0$ q.e. in $A^\varepsilon$ and in $\Omega\setminus(A_k\cup A^\varepsilon)$, so that $w_k\in W^{s,p}_0(A_k)$. Moreover, since the double integral in \eqref{J} only involves differences, \eqref{submodular} gives $J_{s,p}(w_k)\le J_{s,p}(1-v_\varepsilon)+J_{s,p}(u^s_{A_k\cup A^\varepsilon})$. Thus $\{w_k\}_{k\in\N}$ is bounded in $W^{s,p}_0(\Omega)$ and $w_k\rightharpoonup\min\{v_\varepsilon,u^\varepsilon\}$ weakly in $W^{s,p}_0(\Omega)$. By Lemma \ref{lemaA}, $\min\{v_\varepsilon,u^\varepsilon\}=0$ q.e. in $\{u=0\}$. Since $0\le u\le u^s_A$, we have $\Omega\setminus A\subset\{u^s_A=0\}\subset\{u=0\}$ up to a set of zero capacity, and therefore $\min\{v_\varepsilon,u^\varepsilon\}=0$ q.e. in $\Omega\setminus A$. As $v_\varepsilon=1$ there, $u^\varepsilon=0$ q.e. in $\Omega\setminus A$.

	Finally, $u^s_{A_k\cup A^\varepsilon}\in\K_s$ for every $k$ and, by Lemma \ref{lemmaK} below (whose proof does not depend on the present lemma), $u^\varepsilon\in\K_s$, that is, $(-\Delta_p)^s u^\varepsilon\leq 1$ in $\Omega$. By Lemma \ref{lemma2}, $u^\varepsilon\leq  u_A^s$, and the proof is complete.
\end{proof}

\section{Statement and proof of the main result}\label{sec_teo1}
In this section, based on \cite{BUTTAZZO-DALMASO-93,BONDER-RITTORTO-18,FUSCO-19}, we state and prove the main result.
\begin{thm}\label{teo1}
	Let $0 < s < 1<p<\infty $ and let $\Omega \subset \mathbb{R}^n$ be a bounded open set. Let $F_s \colon \A_{s,p}(\Omega) \to \R$ be such that
	\begin{itemize}
		\item[$(H_1)$] $F_s$ is lower semicontinuous with respect to the $\gamma$-convergence; that is, $A_k \xrightarrow{\gamma} A$ implies $F_s(A) \leq \liminf_{k\to\infty} F_s(A_k)$;
		\item[$(H_2)$] $F_s$ is decreasing with respect to set inclusion; that is, $F_s(A) \geq F_s(B)$ whenever $A \subset B$.
	\end{itemize}
	 Then, for every $0 < c < |\Omega|$, the problem
	\begin{equation}\label{P}
		\min\{F_s(A)\colon A \in \A_{s,p}(\Omega),\ |A| \leq c\}
	\end{equation} has a solution.
\end{thm}
%
\subsection{Proof of Theorem \ref{teo1}}
The idea of the proof is as follows. Given $A\in \A_{s,p}(\Omega)$, we proved in Lemma \ref{lemma2} that $u_A^s$ is the maximum of
$$
\{w\in W^{s,p}_0(\Omega)\colon w\leq 0 \text{ q.e. in } \Omega\setminus A,\ (-\Delta_p)^sw\leq 1\text{ in }\Omega\}.
$$
Moreover, $u_A^s$ belongs to
\begin{equation}\label{K_S}
\K_s=\{w \in W_0^{s,p}(\Omega) \colon w \geq 0,\ (-\Delta_p)^s w \leq 1 \text{ in } \Omega\},
\end{equation}
a set that we define for every $0<s\le1$. We will define a functional $G$ on $\K_s$ satisfying
\begin{itemize}
	\item [$(G_1)$] $G$ is decreasing on $\K_s$;
	\item [$(G_2)$] $G$ is lower semicontinuous on $\K_s$ with respect to the strong topology of $L^p(\Omega)$;
	\item [$(G_3)$] $G( u_A^s)=F_s(A)$ for every $A\in \A_{s,p}(\Omega)$ with $|A|\le c$,
\end{itemize}
and we will conclude that the problem
$$
\min\{G(w)\colon w\in \K_s,\ |\{w>0\}|\le c\}
$$
has a solution $w_0$. If we denote $A_0 = \{w_0 > 0\}$, then $u^s_{A_0}$ is also a minimum point of $G$ over $\K_s$ subject to the condition $|\{w > 0\}| \leq c$ and, hence, $A_0$ is a minimizer of $F_s$ in $\A_{s,p}(\Omega)$ subject to the condition $|A| \leq c$.

First, we prove some properties of the set $\K_s$.
\begin{lema}\label{lemma_charK}
	Let $0<s\le1$ and $w\in W^{s,p}_0(\Omega)$. Then $(-\Delta_p)^sw\le1$ in $\Omega$ if and only if $I_s(w)\le I_s(z)$ for every $z\in W^{s,p}_0(\Omega)$ with $z\le w$.
\end{lema}
\begin{proof}
	Assume that $(-\Delta_p)^sw\le1$ in $\Omega$ and let $z\le w$. By the convexity of $J_{s,p}$ and since $w-z\ge0$,
	$$
	I_s(z)-I_s(w)\ge\langle J'_{s,p}(w),z-w\rangle-\int_\Omega(z-w)\,dx=\int_\Omega(w-z)\,dx-\langle J'_{s,p}(w),w-z\rangle\ge0.
	$$
	Conversely, given $\varphi\in W^{s,p}_0(\Omega)$ with $\varphi\ge0$ and $t>0$, the function $w-t\varphi$ satisfies $w-t\varphi\le w$, so $I_s(w-t\varphi)\ge I_s(w)$. Dividing by $t$ and letting $t\downarrow0$, we get $-\langle J'_{s,p}(w),\varphi\rangle+\int_\Omega\varphi\,dx\ge0$.
\end{proof}

\begin{lema}\label{lemmaK}
	Let $0<s<1$. Given any sequence $\{u_k\}_{k\in\N}\subset \K_s$, there exist $u\in \K_s$ and a subsequence $\{u_{k_j}\}_{j\in\N}$ such that $u_{k_j}\to u$ in $L^p(\Omega)$ and weakly in $W^{s,p}_0(\Omega)$.

	Furthermore, if $u,v\in\K_s$, then $\max\{u, v\}\in\K_s$.
\end{lema}
\begin{proof}
	Let $\{u_k\}_{k\in\N}\subset \K_s$. Since $u_k\ge0$ and $(-\Delta_p)^su_k\leq 1$ in $\Omega$, taking $u_k$ as test function and using the Poincar\'e inequality we get $pJ_{s,p}(u_k)\le\int_\Omega u_k\,dx\le CJ_{s,p}(u_k)^{1/p}$, so $\{u_k\}_{k\in\N}$ is bounded in $W^{s,p}_0(\Omega)$. Hence there exist a subsequence, still denoted by $\{u_k\}_{k\in\N}$, and $u\in W^{s,p}_0(\Omega)$ such that $u_k\rightharpoonup u$ weakly in $W^{s,p}_0(\Omega)$, $u_k\to u$ in $L^p(\Omega)$ and a.e. in $\R^n$. In particular, $u\ge0$.

	Now let
	 $$
	\eta_k(x,y)=\frac{|u_k(x)-u_k(y)|^{p-2}(u_k(x)-u_k(y))
}{|x-y|^{\frac{n+sp}{p'}}}.
	$$
	Since $\|\eta_k\|_{L^{p'}(\R^n\times\R^n)}^{p'}=[u_k]_{s,p}^p$, the sequence $\{\eta_k\}_{k\in\N}$ is bounded in $L^{p'}(\R^n\times \R^n)$, where $p'$ is the conjugate exponent of $p$, and, up to a subsequence, $\eta_k \rightharpoonup \eta$ weakly in $L^{p'}(\R^n \times \R^n)$. Since $u_k \rightarrow u$ a.e. in $\R^n$,
	$$\eta_k(x,y) \rightarrow \frac{|u(x) - u(y)|^{p-2}(u(x) - u(y))}{ |x-y|^{\frac{n+sp}{p'}}}$$
	a.e. in $\R^n \times \R^n$, and therefore
	\[
	\eta(x, y) = \frac{|u(x) - u(y)|^{p-2}(u(x) - u(y))}{|x - y|^{\frac{n+sp}{p'}}} .
	\]
	Consequently, for every $v\in W^{s,p}_0(\Omega)$, since $(v(x)-v(y))|x-y|^{-\frac np-s}\in L^p(\R^n\times\R^n)$,
	$$
	\langle J_{s,p}'(u),v\rangle =c_{n,p,s}\iint_{\R^n \times \R^n} \eta(x, y) \frac{v(x) - v(y)}{|x - y|^{\frac{n}{p}+s}} \, dx \, dy=\lim_{k\to\infty}\langle J_{s,p}'(u_k),v\rangle.
	$$
	In particular, $\langle J'_{s,p}(u),v\rangle\le\int_\Omega v\,dx$ for every $v\ge0$. Hence $u\in \K_s$, which concludes the first part of the proof.

	Now let $u,v\in\K_s$, let $w=\max\{u,v\}$, and consider the closed convex set
	$$
	E=\{z\in W^{s,p}_0(\Omega)\colon z\leq w \text{ in }\Omega\}
	$$
	and the unique minimizer $z_0$ of $I_s$ on $E$. It satisfies the variational inequality
	\begin{equation}\label{eq2}
	\langle J_{s,p}'(z_0) , z-z_0\rangle
	\geq \int_{\Omega} (z-z_0)\,dx\quad\forall z\in E.
	\end{equation}
	Since every $z\le z_0$ belongs to $E$, $z_0$ minimizes $I_s$ among the functions $z\le z_0$, and by Lemma \ref{lemma_charK} we get $(-\Delta_p)^sz_0\leq 1$ in $\Omega$.

	Next we prove that $z_0 \geq u$ (and, by symmetry, $z_0 \geq v$), from which $z_0 \geq w$ follows. Since $z_0 \in E$, the reverse inequality also holds, so $z_0 = w$ and, as $w\ge0$, $w\in \mathcal{K}_s$. Let $\zeta=\max\{z_0,u\}\in E$, so that $\zeta-z_0=(u-z_0)^+$. Using $\zeta$ as test function in \eqref{eq2} and the fact that $(-\Delta_p)^su\leq 1$ in $\Omega$, we get
	$$
	\langle J'_{s,p}(u),(u-z_0)^+\rangle\leq \int_{\Omega}(u-z_0)^+\, dx\le\langle J'_{s,p}(z_0),(u-z_0)^+\rangle.
	$$
	By Lemma \ref{lemma_comparison}, $u\le z_0$ a.e.
\end{proof}
Now we can return to the proof of Theorem \ref{teo1}.
\begin{proof}[Proof of Theorem \ref{teo1}]
To simplify the notation, throughout the proof we write $F=F_s$ and $u_A=u^s_A$. Let
$$
\A_{s,p}^c(\Omega)=\{A\in \A_{s,p}(\Omega)\colon |A|\leq c\}
$$
and, given $w\in \mathcal{K}_s$, define
\begin{equation}\label{m}
m(w)=\inf\{F(A)\colon A\in \A_{s,p}^c(\Omega),\ u_A\leq w\}
\end{equation}
(the set in \eqref{m} is nonempty, since $u_\emptyset=0\le w$). The functional $m$ need not be lower semicontinuous; therefore, we define $G$ as the lower semicontinuous envelope of $m$ in $\mathcal{K}_s$ with respect to the strong topology of $L^p(\Omega)$, i.e.,
\begin{equation}\label{eq3}
G(w)=\inf\Big\{\liminf_{k\to\infty} m(w_k)\Big\},
\end{equation}
where the infimum is taken over all sequences $\{w_k\}_{k\in \N}\subset\mathcal{K}_s$ such that $w_k\to w$ in $L^p(\Omega)$. By construction, $G$ satisfies $(G_2)$, and for every $w\in\K_s$ there is a sequence $\{w_k\}_{k\in\N}\subset\K_s$ such that $w_k\to w$ in $L^p(\Omega)$ and $m(w_k)\to G(w)$.

Let us see that $G$ is decreasing in $\K_s$. Let $u,v\in \K_s$ be such that $u\leq v$ and let $\{u_k\}_{k\in\N}\subset \K_s$ be such that $u_k\to u$ in $L^p(\Omega)$ and $m(u_k)\to G(u)$.
By Lemma \ref{lemmaK}, $v_k=\max\{v,u_k\}\in \K_s$ for each $k\in \N$, and $v_k\to \max\{v,u\}=v$ in $L^p(\Omega)$. Consequently, since $m$ is nonincreasing and $v_k\geq u_k$ for every $k\in\N$, we get
$$
G(v)\leq \liminf_{k\to\infty}m(v_k)\leq \lim_{k\to\infty}m(u_k)=G(u).
$$
Now, let us see that $G$ verifies $(G_3)$, i.e.,
\begin{equation}\label{G3}
G( u_A)=F(A)\quad\text{for every }A\in\A_{s,p}^c(\Omega).
\end{equation}
Given $A\in \A_{s,p}^c(\Omega)$, we have $G(u_A)\le m(u_A)\leq F(A)$. To prove the reverse inequality, it suffices to show that
$$
F(A)\leq \liminf_{k\to\infty}m(w_k)
$$
for every sequence $\{w_k\}_{k\in\N}\subset\K_s$ such that $w_k\to  u_A$ in $L^p(\Omega)$; passing to a subsequence, we may assume that the $\liminf$ is a limit.
By definition of $m$, there exist $A_k\in \A_{s,p}^c(\Omega)$ such that
$$
F(A_k)\leq m(w_k)+\frac{1}{k}\quad\text{and}\quad u_{A_k}\leq w_k.
$$
Since $u_{A_k}\in \K_s$ for each $k$, by Lemma \ref{lemmaK}, up to a subsequence, there exists $u\in \K_s$ such that $u_{A_k}\rightharpoonup u$ weakly in $W^{s,p}_0(\Omega)$ and strongly in $L^p(\Omega)$.
Since $w_k\to  u_A$ in $L^p(\Omega)$ and $u_{A_k}\leq w_k$, we get $u\leq  u_A$.

Given $\varepsilon>0$, consider the quasi-open set $A^\varepsilon=\{ u_A>\varepsilon\}$. Since $u_{A_k\cup A^\varepsilon}\in \K_s$, by Lemma \ref{lemmaK}, up to a subsequence, $u_{A_k\cup A^\varepsilon}\rightharpoonup u^\varepsilon$ weakly in $W^{s,p}_0(\Omega)$ for some $u^\varepsilon \in \K_s$. Therefore, by Lemma \ref{lemmaB}, $u^\varepsilon\leq  u_A$.

Next, let us see that $(u_A-\varepsilon)^+\leq u_{A^\varepsilon}$. Write $w=(u_A-\varepsilon)^+$. We have
$$
w(x)-w(y)=\begin{cases}
 u_A(x)- u_A(y)&\text{if } x,y\in A^\varepsilon,\\
 u_A(x)-\varepsilon&\text{if } x\in A^\varepsilon\text{ and } y\not\in A^\varepsilon,\\
\varepsilon- u_A(y)&\text{if } x\not\in A^\varepsilon\text{ and } y\in A^\varepsilon,\\
0&\text{otherwise}.
\end{cases}
$$
Let $v\in W^{s,p}_0(A^\varepsilon)$ with $v\geq 0$, so that $v=0$ a.e. outside $A^\varepsilon$. If $x\in A^\varepsilon$ and $y\notin A^\varepsilon$, then $u_A(y)\le\varepsilon$, so $w(x)-w(y)\le u_A(x)-u_A(y)$ and $v(x)-v(y)=v(x)\ge0$; if $x\notin A^\varepsilon$ and $y\in A^\varepsilon$, then $w(x)-w(y)\ge u_A(x)-u_A(y)$ and $v(x)-v(y)=-v(y)\le0$. Since $t\mapsto|t|^{p-2}t$ is increasing, in all cases
\begin{multline*}
|w(x)-w(y)|^{p-2}(w(x)-w(y))(v(x)-v(y))\\
\le|u_A(x)-u_A(y)|^{p-2}(u_A(x)-u_A(y))(v(x)-v(y)),
\end{multline*}
and integrating we obtain
$$
\langle J'_{s,p}(w),v\rangle\le\langle J'_{s,p}(u_A),v\rangle=\int_\Omega v\,dx=\langle J'_{s,p}(u_{A^\varepsilon}),v\rangle,
$$
where we used that $W^{s,p}_0(A^\varepsilon)\subset W^{s,p}_0(A)$. Hence $(-\Delta_p)^s( u_A-\varepsilon)^+\leq(-\Delta_p)^su_{A^\varepsilon}$ in $A^\varepsilon$. Moreover, $(u_A-\varepsilon)^+=0$ outside $A^\varepsilon$ and $u_{A^\varepsilon}\ge0$, so $((u_A-\varepsilon)^+-u_{A^\varepsilon})^+\in W^{s,p}_0(A^\varepsilon)$. By Lemma \ref{lemma_comparison}, $( u_A-\varepsilon)^+\leq u_{A^\varepsilon}$.

In conclusion, using also Lemma \ref{lemma2}, we have obtained the chain of inequalities
$$
( u_A-\varepsilon)^+\leq u_{A^\varepsilon}\leq u_{A_k\cup A^\varepsilon}.
$$
Since $u_{A_k\cup A^\varepsilon}\to u^\varepsilon$ in $L^p(\Omega)$ and $u^\varepsilon\leq  u_A$, we conclude that
$$
( u_A-\varepsilon)^+\leq u^\varepsilon\leq  u_A,
$$
and in particular $\|u^\varepsilon-u_A\|_{L^p(\Omega)}\le\varepsilon|\Omega|^{1/p}$. Since this holds for every limit point $u^\varepsilon$ of $\{u_{A_k\cup A^\varepsilon}\}_{k\in\N}$, we get $\limsup_{k}\|u_{A_k\cup A^\varepsilon}-u_A\|_{L^p(\Omega)}\le\varepsilon|\Omega|^{1/p}$. Hence, by a diagonal argument, there exists a sequence $\varepsilon_k\downarrow0$ such that
$$
u_{A_k\cup A^{\varepsilon_k}} \to  u_A \quad\text{in }L^p(\Omega).
$$
By Definition \ref{def gamma}, $A_k\cup A^{\varepsilon_k}$ $\gamma$-converges to $A$ and, by $(H_1)$ and $(H_2)$, we obtain the desired inequality
$$
F(A)\leq \liminf_{k\to\infty}F(A_k\cup A^{\varepsilon_k})\leq \liminf_{k\to\infty}F(A_k)\leq \lim_{k\to\infty}m(w_k).
$$

Now, let us focus on the final part of the proof.
Let $\{w_k\}_{k\in\N}\subset \K_s$ be a sequence such that $|\{w_k>0\}|\leq c$ and
$$
\lim_{k\to\infty}G(w_k)=\inf\{G(w)\colon w\in \K_s,\ |\{w>0\}|\leq c\}=:g.
$$
By Lemma \ref{lemmaK}, up to a subsequence, there exists $w_0\in \K_s$ such that $w_k\to w_0$ in $L^{p}(\Omega)$ and a.e. in $\Omega$. Then $\chi_{\{w_0>0\}}\le\liminf_k\chi_{\{w_k>0\}}$ a.e., and by Fatou's lemma $|\{w_0>0\}|\leq c$. Since $G$ is lower semicontinuous on $\K_s$,
$$
g\leq G(w_0)\leq \liminf_{k\to\infty}G(w_k)=g,
$$
so $w_0$ verifies
\begin{equation}\label{eq4}
G(w_0)=\min\{G(w)\colon w\in \K_s,\ |\{w>0\}|\leq c\}.
\end{equation}
Finally, consider the set $A_0=\{w_0>0\}$, which is quasi-open by Remark \ref{rem_capacity}(d). Then $A_0\in \A_{s,p}^{c}(\Omega)$ and, by Lemma \ref{lemma2}, $w_0\leq u_{A_0}$.
For every $A\in \A_{s,p}^c(\Omega)$ we know that $u_A\in \K_s$ and $|\{ u_A>0\}|\le|A|\leq c$. Then, since $G$ is decreasing and by \eqref{G3} and \eqref{eq4}, we have
$$
F(A_0)=G(u_{A_0})\leq G(w_0)\leq G( u_A)=F(A).
$$
Therefore, $A_0$ is a solution of problem \eqref{P}, and the proof is complete.
\end{proof}

\section{Asymptotic behavior as $s\nearrow 1$}\label{sec_asint}
In this section, based on \cite{BBM-01,BONDER-SALORT-20}, we analyze the behavior of the problem as $s\nearrow 1$. To perform this analysis we need to assume some asymptotic behavior of the cost functionals $F_s$ and, to this end, we need a notion of convergence of sets when $s$ varies.

We first recall the notion of $\Gamma$-convergence (see \cite{DALMASO-93}).
\begin{defn}
	Let $(X,d)$ be a metric space and $F_k,F\colon X\to \overline{\R}$, $k\in \N$. We say that $F_k$ $\Gamma$-converges to $F$ if
	\begin{itemize}
		\item for every $x\in X$ and every sequence $\{x_k\}_{k\in\N}$ such that $x_k\to x$, it holds that
		\begin{equation*}
			F(x)\leq\liminf_{k\to\infty}F_k(x_k);
		\end{equation*}
		\item for every $x\in X$, there exists a sequence $\{y_k\}_{k\in\N}$ such that $y_k\to x$ and
		\begin{equation*}
			F(x)\geq\limsup_{k\to\infty}F_k(y_k).
		\end{equation*}
	\end{itemize}
\end{defn}
We will use the following facts, valid for every sequence $0<s_k\nearrow1$.
\begin{enumerate}[label=(F\arabic*)]
	\item\label{F1} There exist $s_0\in(0,1)$ and $C_P>0$, depending only on $n$, $p$ and $\Omega$, such that $\|u\|^p_{L^p(\Omega)}\le C_PJ_{s,p}(u)$ for every $s\in[s_0,1]$ and $u\in W^{s,p}_0(\Omega)$. Since $J_{s,p}(u)=\frac{1-s}{K(n,p)}[u]_{s,p}^p$, this is the bound $C(\Omega,p,s)\le(1-s)C(\Omega,p)$ for the optimal constant in $\|u\|^p_{L^p(\Omega)}\le C(\Omega,p,s)[u]^p_{s,p}$, see \cite{BRASCO-PARINI-16}.
	\item\label{F2} $J_{s_k,p}(u)\to J_{1,p}(u)$ for every $u\in W^{1,p}_0(\Omega)$ \cite{BBM-01}.
	\item\label{F3} If $u_k\in W^{s_k,p}_0(\Omega)$ and $\sup_kJ_{s_k,p}(u_k)<\infty$, then there exist $u\in W^{1,p}_0(\Omega)$ and a subsequence $\{u_{k_j}\}_{j\in\N}$ such that $u_{k_j}\to u$ in $L^p(\Omega)$ \cite[Theorem 4]{BBM-01}. (The limit belongs to $W^{1,p}(\R^n)$ and vanishes outside $\Omega$, hence it belongs to $W^{1,p}_0(\Omega)$ since $\partial\Omega$ is Lipschitz.)
	\item\label{F4} $J_{s_k,p}$ $\Gamma$-converges to $J_{1,p}$ in $L^p(\Omega)$ \cite{PONCE-04} (both functionals being extended by $+\infty$ outside their domains); in particular, $J_{1,p}(u)\le\liminf_kJ_{s_k,p}(u_k)$ whenever $u_k\to u$ in $L^p(\Omega)$.
\end{enumerate}
We also need the following remark.
\begin{rem}\label{rem_inclusion}
	Let $0<s<1$. For $u\in W^{1,p}_0(\Omega)$ we have $[u]_{s,p}\le C\|\nabla u\|_{L^p(\Omega)}$, with $C$ depending on $n$, $p$, $s$ and $\Omega$. Hence $\capa(s,p)\le C\capa(1,p)$, so every $(1,p)$-quasi-open set is $(s,p)$-quasi-open, every $(1,p)$-quasi-continuous function is $(s,p)$-quasi-continuous, and properties that hold $(1,p)$-q.e. also hold $(s,p)$-q.e. In particular, $\A_{1,p}(\Omega)\subset\A_{s,p}(\Omega)$ and $W^{1,p}_0(A)\subset W^{s,p}_0(A)$ for every $A\in\A_{1,p}(\Omega)$.
\end{rem}

\begin{defn}
	Let $0<s_k\nearrow1$, $A_k\in\A_{s_k,p}(\Omega)$ and $A\in \A_{1,p}(\Omega)$. We say that
	$A_k\xrightarrow{\gamma,k}A$ if $u_{A_k}^{s_k}\to  u_A$ strongly in $L^p(\Omega)$.
	\end{defn}
\begin{thm}\label{teo2}
Let $0<c<|\Omega|$ and, for $0<s\leq1$, let $F_s\colon \A_{s,p}(\Omega)\to \R$ be functionals such that $F_s$ satisfies $(H_1)$ and $(H_2)$ for every $0<s<1$, and
\begin{itemize}
	\item [$(H_3)$] continuity with respect to $s$: for every $A \in \A_{1,p}(\Omega)$,
	$$
	F_1(A) = \lim_{s \uparrow 1} F_s(A); $$
	\item[$(H_4)$] for every $0 < s_k \uparrow 1$ and every sequence $A_k\in\A_{s_k,p}(\Omega)$ such that $A_k\xrightarrow{\gamma,k}A$,
	\[ F_1(A) \leq \liminf_{k\to\infty} F_{s_k}(A_k). \]
\end{itemize}
Then the minimum on the left-hand side below is attained, and
$$
\min\{F_1(A)\colon A\in \A_{1,p}(\Omega),\ |A|\leq c\}=\lim_{s\nearrow 1}\min\{F_s(A)\colon A\in \A_{s,p}(\Omega),\ |A|\leq c\}.
$$
\end{thm}
Note that $(H_3)$ makes sense by Remark \ref{rem_inclusion}.

%
%
%
Our first aim is to show that a sequence $\{u_k\}_{k\in\N}\subset L^p(\Omega)$ such that $u_k\in \K_{s_k}$ is precompact and that every accumulation point belongs to $\K_1$.

\begin{lema}\label{lemma1}
	Let $0<s_k\nearrow 1$ and let $u_k\in \K_{s_k}$. Then there exist $u\in\K_1$ and a subsequence $\{u_{k_j}\}_{j\in \N}$ such that $u_{k_j}\to u$ in $L^p(\Omega)$, where
	$$
	\K_1=\{w\in W^{1,p}_0(\Omega)\colon w\geq0,\ -\Delta_pw\leq1\text{ in }\Omega\}.
	$$

\end{lema}
\begin{proof}
	We may assume that $s_k\ge s_0$ for every $k$. Testing $(-\Delta_p)^{s_k}u_k\le1$ with $u_k\ge0$ and using \ref{F1}, we get
	$$
	pJ_{s_k,p}(u_k)\le\int_\Omega u_k\,dx\le|\Omega|^{1/p'}\big(C_PJ_{s_k,p}(u_k)\big)^{1/p},
	$$
	so $\sup_kJ_{s_k,p}(u_k)<\infty$. By \ref{F3}, there exist $u\in W^{1,p}_0(\Omega)$ and a subsequence, still denoted by $\{u_k\}_{k\in\N}$, such that $u_k\to u$ in $L^p(\Omega)$ and a.e.; in particular, $u\ge0$.

	It remains to show that $-\Delta_pu\le1$ in $\Omega$. By Lemma \ref{lemma_charK} (with $s=1$), it suffices to show that $I_1(u)\le I_1(z)$ for every $z\in W^{1,p}_0(\Omega)$ with $z\le u$. Fix such a $z$. Since $W^{1,p}_0(\Omega)\subset W^{s_k,p}_0(\Omega)$ and $\min\{z,u_k\}\le u_k$, Lemma \ref{lemma_charK} (with $s=s_k$) gives $I_{s_k}(u_k)\le I_{s_k}(\min\{z,u_k\})$. On the other hand, by \eqref{submodular},
	$$
	J_{s_k,p}(\min\{z,u_k\})+J_{s_k,p}(\max\{z,u_k\})\le J_{s_k,p}(z)+J_{s_k,p}(u_k).
	$$
	Combining both inequalities we obtain
	$$
	J_{s_k,p}(\max\{z,u_k\})\le J_{s_k,p}(z)+\int_\Omega u_k\,dx-\int_\Omega\min\{z,u_k\}\,dx.
	$$
	Since $\max\{z,u_k\}\to\max\{z,u\}=u$ in $L^p(\Omega)$, by \ref{F4} and \ref{F2} we can pass to the limit and get $J_{1,p}(u)\le J_{1,p}(z)+\int_\Omega u\,dx-\int_\Omega z\,dx$, that is, $I_1(u)\le I_1(z)$.
	\end{proof}
Next, we analyze the behavior of $u_A^s$ as $s$ approaches 1.
\begin{lema}\label{lemma3}
	Let $A\in \A_{1,p}(\Omega)$ and $0<s_k\nearrow1$. If $u_A^{s_k}\to \bar u$ in $L^p(\Omega)$, then $\bar u\ge u_A$.
\end{lema}
\begin{proof}
	By Remark \ref{rem_inclusion}, $A\in\A_{s_k,p}(\Omega)$ and $u_A\in W^{1,p}_0(A)\subset W^{s_k,p}_0(A)$, so $u^{s_k}_A$ is well defined; moreover, by Lemma \ref{lemma1}, $\bar u\in\K_1$, and in particular $\bar u\ge0$. Since $\max\{u^{s_k}_A,u_A\}\in W^{s_k,p}_0(A)$ and $u^{s_k}_A$ minimizes $I_{s_k}$ in $W^{s_k,p}_0(A)$, using \eqref{submodular} and the identity $\int\min\{a,b\}+\int\max\{a,b\}=\int a+\int b$, we get
	$$
	I_{s_k}(\min\{u^{s_k}_A,u_A\})\le I_{s_k}(u^{s_k}_A)+I_{s_k}(u_A)-I_{s_k}(\max\{u^{s_k}_A,u_A\})\le I_{s_k}(u_A).
	$$
	Since $\min\{u^{s_k}_A,u_A\}\to\min\{\bar u,u_A\}$ in $L^p(\Omega)$, by \ref{F4} and \ref{F2},
	$$
	I_1(\min\{\bar u,u_A\})\le\liminf_{k\to\infty}I_{s_k}(\min\{u^{s_k}_A,u_A\})\le\lim_{k\to\infty}I_{s_k}(u_A)=I_1(u_A).
	$$
	Finally, $0\le\min\{\bar u,u_A\}\le u_A$, so $\min\{\bar u,u_A\}\in W^{1,p}_0(A)$ by Remark \ref{rem_capacity}(c). Since $u_A$ is the unique minimizer of $I_1$ in $W^{1,p}_0(A)$, we conclude that $\min\{\bar u,u_A\}=u_A$, that is, $\bar u\ge u_A$.
\end{proof}
Lemma \ref{lemma3} only provides a lower bound for the limit of $u^{s}_A$, which is all that is needed in the proof of Proposition \ref{prop1} below.

The following two lemmas are the counterparts of Lemmas \ref{lemaA} and \ref{lemmaB}.
\begin{lema}\label{lemmaA2}
	Let $0<s_k\nearrow1$ and, for every $k\in \N$, let $A_k\in \A_{s_k,p}(\Omega)$ be such that $u_{A_k}^{s_k}\to u$ strongly in $L^p(\Omega)$. Let $w_k\in W^{s_k,p}_0(A_k)$ be such that $\sup_{k\in \N}J_{s_k,p}(w_k)<\infty$, and assume that $w_k\to w$ strongly in $L^p(\Omega)$. Then $w\in W^{1,p}_0(\{u>0\})$.
\end{lema}
\begin{proof}
	By Lemma \ref{lemma1} and \ref{F3}, $u\in\K_1$ and $w\in W^{1,p}_0(\Omega)$; here $u$ and $w$ denote their $(1,p)$-quasi-continuous representatives, so that $\{u>0\}\in\A_{1,p}(\Omega)$, and we have to show that $w\in\mathcal C:=\{z\in W^{1,p}_0(\Omega)\colon z=0\ (1,p)\text{-q.e. in }\{u=0\}\}$. We follow the proof of Lemma \ref{lemaA}, with the following changes. We may assume that $s_k\ge s_0$, so that, by \ref{F1}, the constant in Lemma \ref{lemma_stability} does not depend on $k$; moreover, weak compactness in $W^{s,p}_0(\Omega)$ and weak lower semicontinuity of $J_{s,p}$ are replaced by \ref{F3} and \ref{F4}, respectively.

	More precisely, let $L=\sup_kJ_{s_k,p}(w_k)$ and, for $\lambda>0$, let $\hat w_k$ be the minimizer of $J_{s_k,p}(z)+\frac{1}{p\lambda}\|z-w\|_p^p$ in $W^{s_k,p}_0(A_k)$. As in Step 1 of the proof of Lemma \ref{lemaA}, $\sup_kJ_{s_k,p}(\hat w_k)<\infty$ and, by \ref{F3} and \ref{F4}, up to a subsequence, $\hat w_k\to w^\lambda$ in $L^p(\Omega)$, with $w^\lambda\in W^{1,p}_0(\Omega)$ and
	$$
	J_{1,p}(w^\lambda)+\frac{1}{p\lambda}\|w^\lambda-w\|_p^p\le L.
	$$
	Moreover, $\hat w_k$ solves $(-\Delta_p)^{s_k}\hat w_k=f_k$ in $A_k$, with $f_k\to f:=-\lambda^{-1}|w^\lambda-w|^{p-2}(w^\lambda-w)$ in $L^{p'}(\Omega)$. Given $\delta>0$ and $f_\delta\in L^\infty(\Omega)$ with $\|f_\delta-f\|_{p'}\le\delta$, let $\omega_{k,\delta}$ be the solution in $W^{s_k,p}_0(A_k)$ of $(-\Delta_p)^{s_k}\omega_{k,\delta}=f_\delta$ in $A_k$. By Lemma \ref{lemma_stability}, $\sup_kJ_{s_k,p}(\omega_{k,\delta})<\infty$ and $J_{s_k,p}(\omega_{k,\delta}-\hat w_k)\le\rho(\|f_\delta-f_k\|_{p'})$, with $\rho$ independent of $k$. By \ref{F3} and \ref{F4}, up to a subsequence, $\omega_{k,\delta}\to\omega_\delta$ in $L^p(\Omega)$, with $\omega_\delta\in W^{1,p}_0(\Omega)$ and $J_{1,p}(\omega_\delta-w^\lambda)\le\rho(\delta)$. As before, Lemma \ref{lemma_comparison} gives $|\omega_{k,\delta}|\le c_\delta u^{s_k}_{A_k}$ a.e., hence $|\omega_\delta|\le c_\delta u$ a.e. and, by Remark \ref{rem_capacity}(c) with $s=1$, $(1,p)$-q.e. Thus $\omega_\delta\in\mathcal C$ and, letting $\delta\to0$, $w^\lambda\in\mathcal C$. Finally, $\|w^\lambda-w\|_p^p\le p\lambda L$ and $J_{1,p}(w^\lambda)\le L$, so $w^\lambda\rightharpoonup w$ weakly in $W^{1,p}_0(\Omega)$ as $\lambda\to0$, and therefore $w\in\mathcal C$.
	\end{proof}
\begin{lema}\label{lemmaB2}
	Let $0<s_k\nearrow1$, $A_k\in \A_{s_k,p}(\Omega)$ for every $k\in\N$, and $A\in\A_{1,p}(\Omega)$. Assume that $u_{A_k}^{s_k}\to u$ in $L^p(\Omega)$ and that $u\leq u_A$. Given $\varepsilon>0$, let $A^\varepsilon:=\{u_A>\varepsilon\}$ and assume that $u_{A_k\cup A^\varepsilon}^{s_k}\to u^\varepsilon$ strongly in $L^p(\Omega)$. Then $u^\varepsilon\leq u_A$.
\end{lema}
\begin{proof}
	Note that $A^\varepsilon\in\A_{1,p}(\Omega)\subset\A_{s_k,p}(\Omega)$ by Remarks \ref{rem_capacity} and \ref{rem_inclusion}. By Lemma \ref{lemma1}, $u,u^\varepsilon\in \K_1$; in particular, $u\ge0$. Let us define
	$$
	v_\varepsilon:=1-\frac{1}{\varepsilon}\min\{u_A,\varepsilon\},\qquad u_{k,\varepsilon} := u_{A_k\cup A^\varepsilon}^{s_k},\qquad w_{k,\varepsilon} := \min\{v_\varepsilon, u_{k,\varepsilon}\}.
	$$
	Observe that $0 \leq v_\varepsilon\leq 1$, $1-v_\varepsilon\in W^{1,p}_0(\Omega)$, $v_\varepsilon=0$ in $A^\varepsilon$ and $v_\varepsilon=1$ q.e. in $\Omega\setminus A$. Since $u_{k,\varepsilon}\ge0$ by Lemma \ref{lemma2} and $v_\varepsilon=0$ in $A^\varepsilon$, we have $w_{k,\varepsilon}=0$ in $A^\varepsilon$; moreover, since $u_{k,\varepsilon}=0$ q.e. in $\Omega\setminus(A_k\cup A^\varepsilon)$, also $w_{k,\varepsilon}=0$ there, and consequently $w_{k,\varepsilon}\in W_0^{s_k,p}(A_k)$. By \eqref{submodular}, \ref{F2} and the proof of Lemma \ref{lemma1},
	$$
	\sup_kJ_{s_k,p}(w_{k,\varepsilon})\le\sup_k\big(J_{s_k,p}(1-v_\varepsilon)+J_{s_k,p}(u_{k,\varepsilon})\big)<\infty.
	$$
	Moreover, $w_{k,\varepsilon} \to w_\varepsilon := \min\{v_\varepsilon, u^\varepsilon\}$ strongly in $L^p(\Omega)$. Applying Lemma \ref{lemmaA2}, we get $w_\varepsilon \in W_0^{1,p}(\{u > 0\})$, so $w_\varepsilon = 0$ q.e. in $\{u = 0\}$. The relation $0 \leq u \leq u_A$ implies $\{u_A = 0\} \subset \{u = 0\}$, and $\Omega\setminus A\subset\{u_A=0\}$ up to a set of zero capacity; hence $w_\varepsilon=0$ q.e. in $\Omega\setminus A$. Since $v_\varepsilon = 1$ q.e. in $\Omega \setminus A$, we get $u^\varepsilon = 0$ q.e. in $\Omega \setminus A$. As $u^\varepsilon\in\K_1$, Lemma \ref{lemma2} with $s=1$, namely
	$$
	u_A=\max\{w\in W^{1,p}_0(\Omega)\colon w\leq 0 \text{ q.e. in } \Omega\setminus A\text{ and } -\Delta_pw\leq 1 \text{ in } \Omega\},
	$$
	yields $u^\varepsilon\leq u_A$.
\end{proof}
\begin{prop}\label{prop1}
	Let $0 < s_k \nearrow 1$ and $A_k \in  \A_{s_k,p}(\Omega)$ be such that $u^{s_k}_{A_k} \to u$ in $L^p(\Omega)$. Then $u\in\K_1$, $A := \{u > 0\}\in\A_{1,p}(\Omega)$, and there exist $\tilde{A}_k \in \mathcal{A}_{s_k,p}(\Omega)$ such that $A_k \subset \tilde{A}_k$ and $\tilde{A}_k\xrightarrow{\gamma,k}A$.
	\end{prop}
\begin{proof}
By Lemma \ref{lemma1}, $u\in \K_1$; by Remark \ref{rem_capacity}(d), $A\in\A_{1,p}(\Omega)$; and by Lemma \ref{lemma2} with $s=1$, $u\leq u_A$.

Given $\varepsilon>0$, consider $A^\varepsilon:=\{u_A> \varepsilon\}$ and observe that, by Lemma \ref{lemma2},
$$
u^{s_k}_{A^\varepsilon}\leq u^{s_k}_{A_k\cup A^\varepsilon}.
$$
By Lemma \ref{lemma1}, up to a subsequence, $u^{s_k}_{A_k\cup A^\varepsilon}\to u^\varepsilon$ and $u^{s_k}_{A^\varepsilon}\to\bar u^\varepsilon$ in $L^p(\Omega)$, and by Lemma \ref{lemma3}, $\bar u^\varepsilon\ge u_{A^\varepsilon}$. Passing to the limit in the previous inequality, we conclude that
$$
u_{A^\varepsilon} \leq u^\varepsilon.
$$
Moreover, $u_{A^\varepsilon} = (u_A - \varepsilon)^+$. Indeed, $(u_A-\varepsilon)^+\in W^{1,p}_0(A^\varepsilon)$ and, for every $\varphi\in W^{1,p}_0(A^\varepsilon)\subset W^{1,p}_0(A)$, since $\nabla(u_A-\varepsilon)^+=\nabla u_A$ a.e. in $A^\varepsilon$ and $\nabla\varphi=0$ a.e. outside $A^\varepsilon$,
$$
\int_\Omega|\nabla(u_A-\varepsilon)^+|^{p-2}\nabla(u_A-\varepsilon)^+\cdot\nabla\varphi\,dx=\int_\Omega|\nabla u_A|^{p-2}\nabla u_A\cdot\nabla\varphi\,dx=\int_\Omega\varphi\,dx.
$$
Hence, by Lemma \ref{lemmaB2},
$$
(u_A - \varepsilon)^+ \leq u^\varepsilon \leq u_A.
$$
As in the proof of Theorem \ref{teo1}, this implies that $\limsup_k\|u^{s_k}_{A_k\cup A^\varepsilon}-u_A\|_{L^p(\Omega)}\le\varepsilon|\Omega|^{1/p}$ and, by a diagonal argument, there exists a sequence $0 < \varepsilon_k \downarrow 0$ such that
$$
u^{s_k}_{A_k \cup A^{\varepsilon_k}} \to u_A\quad\text{strongly in } L^p(\Omega).
$$
That is, $\tilde{A}_k:=A_k \cup A^{\varepsilon_k}\xrightarrow{\gamma,k}A$.
\end{proof}
Now we can give the proof of Theorem \ref{teo2}.
\begin{proof}[Proof of Theorem \ref{teo2}]
For $0<s\le1$, let $\A_{s,p}^c(\Omega)=\{A\in\A_{s,p}(\Omega)\colon |A|\le c\}$, and let $m_1=\inf\{F_1(A)\colon A\in\A_{1,p}^c(\Omega)\}$. Let $0<s_k\nearrow1$. By Theorem \ref{teo1}, for each $k$ there exists $A_k\in \A_{s_k,p}^c(\Omega)$ such that
$$
F_{s_k}(A_k)=\min\{F_{s_k}(A)\colon A\in \A_{s_k,p}^c(\Omega)\}=:m_k.
$$
On the one hand, if $A\in \A_{1,p}^c(\Omega)$, then $A\in\A_{s_k,p}^c(\Omega)$ by Remark \ref{rem_inclusion} and, by $(H_3)$,
$$
\limsup_{k\to\infty}m_k\leq \lim_{k\to\infty}F_{s_k}(A)=F_1(A).
$$
Hence
\begin{equation}\label{eq_limsup}
\limsup_{k\to\infty}m_k\leq m_1.
\end{equation}
On the other hand, passing to a subsequence, we may assume that $\liminf_km_k$ is a limit and, by Lemma \ref{lemma1}, that $u_{A_k}^{s_k}\to u$ in $L^p(\Omega)$ and a.e. for some $u\in\K_1$. By Proposition \ref{prop1}, there exist $\tilde{A}_k\in \A_{s_k,p}(\Omega)$ such that $A_k\subset \tilde{A}_k$ and $\tilde{A}_k\xrightarrow{\gamma,k}A:=\{u>0\}$. Since $|\{u^{s_k}_{A_k}>0\}|\le|A_k|\le c$, Fatou's lemma gives $|A|\leq c$, that is, $A\in\A_{1,p}^c(\Omega)$.

From hypotheses $(H_2)$ and $(H_4)$ we conclude that
$$
F_1(A)\leq \liminf_{k\to\infty}F_{s_k}(\tilde{A}_k)\leq \liminf_{k\to\infty} F_{s_k}(A_k),
$$
from which it follows that
\begin{equation}\label{eq_liminf}
m_1\le F_1(A)\leq \liminf_{k\to\infty}m_k.
\end{equation}
Combining \eqref{eq_limsup} and \eqref{eq_liminf}, we get $F_1(A)=m_1=\lim_km_k$. Hence the infimum $m_1$ is attained at $A$ and, since the sequence $s_k\nearrow1$ is arbitrary, the proof is complete.
\end{proof}
\section{Anisotropic case}\label{sec_anis}
In this final section, we state the analogues of Theorems \ref{teo1} and \ref{teo2} in an anisotropic functional setting.

Relying on \cite{CERESABONDER23}, we begin by defining the anisotropic spaces that will be considered in this work. Given $i\in\{1,\dots, n\}$, $s\in(0,1]$ and $p\in (1,\infty)$, we define the $i$-th fractional Sobolev space as
$$
W_i^{s,p}(\R^n)=\{u \in L^p(\R^n) \colon [u]_{s,p,i}< \infty\},
$$
where
\[
[u]_{s,p,i} =
\left \{
\begin{aligned}
&(1-s)s\int_{\R^n}\int_\R\frac{|u(x+he_i)-u(x)|^p}{|h|^{1+sp}}\, dh\, dx\ &\text{ if } \ &s < 1,\\
&\frac{2}{p}\int_{\R^n} |\partial_{x_i} u|^p \, dx\ &\text{ if } \ &s=1.
\end{aligned}
\right .
\]
For convenience, we use the notation
\begin{align*}
	&\D_{s_i,p_i}(u):=\frac{|u(x+he_i)-u(x)|^{p_i}}{|h|^{1+s_ip_i}}, \\ &\D_{s_i,p_i}'(u):=\frac{|u(x+he_i)-u(x)|^{p_i-2}(u(x+he_i)-u(x))}{|h|^{1+s_ip_i}}.
\end{align*}
It is easy to see that $W^{s,p}_i(\R^n)$ is a separable and reflexive Banach space with the norm
$$
\|u\|_{s,p,i} = (\|u\|_p^p + [u]_{s,p,i})^{1/p}.
$$
This space consists of $L^p$ functions that have some mild regularity in the $i$-th variable. In order to control every variable, we proceed as follows.
\begin{defn}\label{s,p}
	Let $\s=(s_1,\dots,s_n)$ and $\p=(p_1,\dots,p_n)$. We define the parameters
	$$
	\overline{\mathbf{s}}=\left(\frac{1}{n}\sum_{i=1}^n\frac{1}{s_i}\right)^{-1},\quad  \overline{\s\p}=\left(\frac{1}{n}\sum_{i=1}^n\frac{1}{s_ip_i}\right)^{-1}\quad\text{and}\quad
	\p^*_{\s} =\frac{n\overline{\s\p}/ \overline{\s}}{n-\overline{\s\p}},
	$$
	and throughout this section we assume that
	\begin{enumerate}
		\item $0<s_i\leq 1$ for all $i\in \{1,\dots, n\}$;
		\item $1<p_1\leq \dots\leq p_n<\infty$;
		\item $\overline{\s\p}<n$ and $p_n<\p^*_\s$.
	\end{enumerate}
	\end{defn}
Then we define
$$
W^{\s,\p}(\R^n)=\bigcap_{i=1}^n W^{s_i,p_i}_i(\R^n).
$$
This anisotropic Sobolev space is a separable and reflexive Banach space with the norm
$$
\|u\|_{\s,\p} = \sum_{i=1}^n \|u\|_{s_i,p_i,i}.
$$
If $\Omega \subset \R^n$ is a bounded open set, we define the anisotropic Sobolev space of functions vanishing in $\R^n\setminus\Omega$ as
\begin{equation}\label{W0}
W^{\s,\p}_0(\Omega)=\{u \in W^{\s,\p}(\R^n) \colon u=0\text{ a.e. in }\R^n\setminus\Omega\}.
\end{equation}
By the Poincar\'e inequality of Proposition \ref{prop_poincare} below, on $W^{\s,\p}_0(\Omega)$ we can use the equivalent norm
$$
\|u\|_{W^{\s,\p}_0(\Omega)} := \sum_{i=1}^{n}[u]_{s_i,p_i,i}^{1/{p_i}}.
$$
\begin{prop}\label{prop_poincare}
	Let $\Omega\subset\R^n$ be a bounded open set, $i\in\{1,\dots,n\}$, $s\in(0,1]$ and $p\in(1,\infty)$. There exists a constant $C=C(\Omega,p,s)>0$ such that
	$$
	\|u\|_{L^p(\Omega)}^p\leq C[u]_{s,p,i}
	$$
	for every $u\in W^{s,p}_i(\R^n)$ with $u=0$ a.e. in $\R^n\setminus\Omega$. Hence there exists $C=C(\Omega,\s, \p)>0$ such that, for every $u \in W^{\s, \p}_0(\Omega)$ and every $i \in \{1, \ldots, n\}$,
	$$
	\|u\|_{L^{p_i}(\Omega)}\leq C\|u\|_{W^{\s,\p}_0(\Omega)}.
	$$
\end{prop}
\begin{proof}
	Let $u$ be as in the statement and let $R>1$ be such that $\Omega\subset Q_R = [-R,R]^n$. If $s=1$, the result follows from the one-dimensional Poincar\'e inequality in the variable $x_i$. If $s<1$, since $u=0$ outside $Q_R$, we have
	\begin{align*}
		\frac{[u]_{s,p,i}}{(1-s)s}&=\int_{\R^n}\int_\R\frac{|u(x+he_i)-u(x)|^{p}}{|h|^{1+sp}}\,dh\,dx\\
		&\geq \int_{Q_R}\int_\R\frac{|u(x+he_i)-u(x)|^{p}}{|h|^{1+sp}}\,dh\,dx\\
		&\geq \int_{Q_R'}\int_{|x_i|\leq R}\int_{|x_i+h|\geq R}\frac{|u(x)|^{p}}{|h|^{1+sp}}\,dh\,dx_i\,dx'\\
		&\geq \int_{Q_R'}\int_{|x_i|\leq R}|u(x)|^{p}\int_{|h|\geq 2R}\frac{1}{|h|^{1+sp}}\,dh\,dx_i\,dx'\\
		&\geq \int_{Q_R'}\int_{|x_i|\leq R}|u(x)|^{p}\int_{|h|\geq 2R}\frac{1}{|h|^{1+p}}\,dh\,dx_i\,dx'\\
		&= C\|u\|_{p}^{p},
	\end{align*}
	where $Q_R' = [-R,R]^{n-1}$, $dx' = dx_1\cdots dx_{i-1}\,dx_{i+1}\cdots dx_n$, and in the last inequality we used that $|h|\ge2R>1$. The second statement follows by applying the first one with $s=s_i$ and $p=p_i$.
\end{proof}

Since we are dealing with admissible domains defined in terms of capacity, we define the analogue of $\capa(s,p)$, namely $\capa(\s,\p)$.
\begin{defn}
	Let $\Omega \subset \mathbb{R}^n$ be an open set, let $\s$ and $\p$ be as in Definition \ref{s,p}, and let $A \subset \Omega$. We define the Gagliardo $(\s,\p)$-capacity of $A$ relative to $\Omega$ as
	$$
	\capa(\s,\p)(A) = \inf \left\{ \|u\|_{W^{\s,\p}_0(\Omega)} \colon u \in W^{\s,\p}_0(\Omega),\ u \geq 0,\ A \subset \{u \geq 1\}^\circ \right\}.
	$$
\end{defn}
In this context, we say that a subset $A$ of $\Omega$ is $(\s,\p)$-quasi-open if there exists a decreasing sequence $\{\omega_k\}_{k\in\mathbb{N}}$ of open subsets of $\Omega$ such that $\capa(\s,\p)(\omega_k) \to 0$ as $k \to +\infty$ and $A \cup \omega_k$ is open for every $k \in \mathbb{N}$. As in Remark \ref{rem_capacity}(a), this is equivalent to the analogue of Definition \ref{def_qo}. We denote by $\A_{\s,\p}(\Omega)$ the class of all $(\s,\p)$-quasi-open subsets of $\Omega$.
\begin{defn}
	A function $u \colon \Omega \rightarrow \R$ is said to be $(\s,\p)$-quasi-continuous (q.c.) if for every $\varepsilon > 0$ there exists a relatively open set $G \subset \Omega$ such that $\capa(\s,\p) (G) < \varepsilon$ and $u|_{\Omega \setminus G}$ is continuous.
\end{defn}
The theorem below allows us to work with the quasi-continuous representative of a function $u\in W^{\s,\p}_0(\Omega)$. Its proof is similar to that of \cite[Theorem 3.7]{WARMA-15}, so we omit it.
\begin{thm}
	For every $u \in W^{\s,\p}_0(\Omega)$ there exists a quasi-continuous function $\tilde{u} \colon \Omega \rightarrow \mathbb{R}$ such that $\tilde{u} = u$ a.e. in $\Omega$, which is unique up to sets of zero $(\s,\p)$-capacity.
\end{thm}
As in Section \ref{sec_preliminares}, for $A\in\A_{\s,\p}(\Omega)$ we let $W^{\s,\p}_0(A)=\{u\in W^{\s,\p}_0(\Omega)\colon u=0\ (\s,\p)\text{-q.e. in }\Omega\setminus A\}$. Finally, since we will work with the fractional pseudo $p$-Laplacian defined in \eqref{fractiona pseudo p}, we give the notion of weak solution.
\begin{defn}\label{weak-solution anis}
	Let $\s$ and $\p$ be as in Definition \ref{s,p}, $A\in\A_{\s,\p}(\Omega)$, and $f \in L^{p_n'}(\Omega)$, where $p_n'=p_n/(p_n-1)$. We say that $u \in W^{\s,\p}_0(A)$ is a weak solution of
	\begin{equation}\label{eq_weak_anis}
	\begin{cases}
	(-\widetilde{\Delta}_\p)^\s u = f &\text{in } A, \\
	u = 0 &\text{q.e. in } \Omega \setminus A,
	\end{cases}
	\end{equation}
	if
	\[
	\langle (-\widetilde{\Delta}_\p)^\s u , v\rangle = \int_\Omega f v \, dx\quad\text{for every } v \in W^{\s,\p}_0(A),
	\]
	where
	$$
	\langle (-\widetilde{\Delta}_\p)^\s u , v\rangle:=\sum_{i=1}^{n}(1-s_i)s_i\iint_{\R^n \times \R} \D_{s_i,p_i}'(u)\,(v(x+he_i)-v(x))\, dh\,dx,
	$$
	with the usual modification $\frac{2}{p_i}\int_{\R^n}|\partial_{x_i}u|^{p_i-2}\partial_{x_i}u\,\partial_{x_i}v\,dx$ for the indices with $s_i=1$.
\end{defn}

As in Section \ref{sec_preliminares}, the strict convexity and coercivity of the associated energy ensure the existence and uniqueness of weak solutions to \eqref{eq_weak_anis}. In particular, given $A \in \A_{\s,\p}(\Omega)$, we denote by $u_A^{\s} \in W_0^{\s,\p}(A)$ the unique weak solution of
\begin{equation}
(-\widetilde{\Delta}_\p)^\s v = 1 \text{ in } A, \qquad v = 0 \text{ q.e. in } \Omega\setminus A.
\end{equation}
The notion of convergence is defined as follows.
\begin{defn} Let $\{A_k\}_{k\in\mathbb{N}} \subset \A_{\s,\p}(\Omega)$ and $A \in \A_{\s,\p}(\Omega)$. We say that $A_k \xrightarrow{\gamma_\s} A$ if $u_{A_k}^{\s} \to  u_A^{\s}$ in $L^{p_1}(\Omega)$.

Moreover, let $\s_k:=(s_1^k,\dots,s_n^k)\nearrow \vec{1}:=(1,\dots,1)$, $A_k\in\A_{\s_k,\p}(\Omega)$ for every $k\in\N$, and $A \in \A_{\vec1,\p}(\Omega)$. We say that $A_k\xrightarrow{\gamma_1} A$ if $u_{A_k}^{\s_k}\to u_A^{\vec{1}}$ in $L^{p_1}(\Omega)$.
\end{defn}

An important tool associated with weak solutions is the comparison principle. While it is well known in the isotropic case (see Lemma \ref{lemma_comparison}), in the anisotropic setting it needs to be stated and proved.
\begin{prop}[Comparison principle]\label{prop_comp_anis}
	Let $D\in\A_{\s,\p}(\Omega)$ and $u, v \in W^{\s,\p}_0(\Omega)$ be such that $(u-v)^+\in W^{\s,\p}_0(D)$ and
	$(-\widetilde{\Delta}_{\p})^{\s}u \leq (-\widetilde{\Delta}_{\p})^{\s}v$ in $D$, in the sense that
	\begin{multline*}
		\sum_{i=1}^{n}(1-s_i)s_i\iint_{\R^n \times \R} \D_{s_i,p_i}'(u)(\phi(x+he_i)-\phi(x))\,dh\,dx \\
		\leq \sum_{i=1}^{n}(1-s_i)s_i\iint_{\R^n \times \R} \D_{s_i,p_i}'(v)(\phi(x+he_i)-\phi(x))\,dh\,dx
	\end{multline*}
	for every $\phi\in W^{\s,\p}_0(D)$ with $\phi\geq0$ (with the usual modification for the indices with $s_i=1$). Then $u \leq v$ a.e. in $\R^n$.
\end{prop}

\begin{proof} The proof is based on \cite{LINDQVIST-14} and follows the lines of Lemma \ref{lemma_comparison}. Let $\psi=u-v$ and take $\phi=\psi^+$, which is admissible. For an index $i$ with $s_i<1$, write $a=u(x+he_i)-u(x)$ and $b=v(x+he_i)-v(x)$, so that $a-b=\psi(x+he_i)-\psi(x)$. Since $t\mapsto|t|^{p_i-2}t$ is strictly increasing and $t\mapsto t^+$ is nondecreasing,
	$$
	\big(|a|^{p_i-2}a-|b|^{p_i-2}b\big)\big(\psi^+(x+he_i)-\psi^+(x)\big)\ge0,
	$$
	with equality only if $\psi^+(x+he_i)=\psi^+(x)$. The same holds for the indices with $s_i=1$, with $\partial_{x_i}u$, $\partial_{x_i}v$ in place of $a$, $b$, and with equality only if $\partial_{x_i}\psi^+=0$. Hence each term in
	$$
	\sum_{i=1}^{n}(1-s_i)s_i\iint_{\R^n \times \R}\left[\D_{s_i,p_i}'(u)-\D_{s_i,p_i}'(v)\right](\psi^+(x+he_i)-\psi^+(x))\,dh\,dx\le0
	$$
	is nonnegative, so all the integrands vanish a.e. Therefore, for every $i$, $\psi^+(x+he_i)=\psi^+(x)$ for a.e. $(x,h)\in\R^n\times\R$ (resp. $\partial_{x_i}\psi^+=0$ a.e.), that is, $\psi^+$ is invariant under translations in every coordinate direction. Thus $\psi^+$ is constant a.e. in $\R^n$ and, since it vanishes outside $\Omega$, $\psi^+=0$, that is, $u\le v$.
\end{proof}
To conclude, equipped with the properties and definitions of the anisotropic case, we can state the anisotropic counterparts of Theorems \ref{teo1} and \ref{teo2}. We omit the proof of Theorem \ref{teo1anis}, since it closely follows the proofs in the preceding sections, with the necessary modifications: Lemma \ref{lemma_comparison} is replaced by Proposition \ref{prop_comp_anis}, the Poincar\'e inequality by Proposition \ref{prop_poincare}, and the results of \cite{BBM-01,PONCE-04} by the anisotropic Bourgain--Brezis--Mironescu-type results of \cite{CERESABONDER23}.
\begin{thm}\label{teo1anis}
	Let $\s,\p$ be as in Definition \ref{s,p} and let $\Omega \subset \R^n$ be a bounded open set. Let $F_\s \colon \A_{\s,\p}(\Omega) \to \R$ be such that
	\begin{itemize}
		\item[$(H^{\s}_1)$] $F_\s$ is lower semicontinuous with respect to the $\gamma_\s$-convergence; that is, $A_k \xrightarrow{\gamma_\s} A$ implies $F_\s(A) \leq \liminf_{k\to\infty} F_\s(A_k)$;
		\item[$(H^{\s}_2)$] $F_\s$ is decreasing with respect to set inclusion; that is, $F_\s(A) \geq F_\s(B)$ whenever $A \subset B$.
	\end{itemize}
	Then, for every $0 < c < |\Omega|$, the problem
	\begin{equation}\label{Panis}
	\min\{F_\s(A)\colon A \in \A_{\s,\p}(\Omega),\ |A| \leq c\}
	\end{equation}
	has a solution. If, moreover, $F_{\vec1}\colon\A_{\vec1,\p}(\Omega)\to\R$ and the family $\{F_\s\}$ satisfy
	\begin{itemize}
		\item [$(H^{\s}_3)$] continuity with respect to $\s$: for every $A \in \A_{\vec1,\p}(\Omega)$,
		$$
		F_{\vec{1}}(A) = \lim_{\s \nearrow \vec1} F_\s(A); $$
		\item[$(H^{\s}_4)$] for every $\s_k \nearrow \vec1$ and every sequence $A_k\in\A_{\s_k,\p}(\Omega)$ such that $A_k \xrightarrow{\gamma_1} A$,
		\[ F_{\vec1}(A) \leq \liminf_{k\to \infty} F_{\s_k}(A_k), \]
	\end{itemize}
	then
	$$
	\min\{F_{\vec1}(A)\colon A\in \A_{\vec1,\p}(\Omega),\ |A|\leq c\}=\lim_{\s\nearrow \vec1}\min\{F_\s(A)\colon A\in \A_{\s,\p}(\Omega),\ |A|\leq c\}.
	$$
\end{thm}



\section*{Acknowledgements}
{\sloppy This work was supported by PICT 2019-00985 and UBACYT No.~20020220200056BA. Ignacio Ceresa Dussel is a doctoral fellow of CONICET.\par}
\bibliographystyle{plain}
\bibliography{References}

\end{document}